\documentclass{amsart}

\usepackage[utf8]{inputenc}

\usepackage{amssymb}
\usepackage{amsmath}
\usepackage{amsthm}
\usepackage{enumitem}
\usepackage{pgf,tikz}
\usetikzlibrary{arrows}
\usepackage[center]{caption}
\usepackage{comment}

\numberwithin{equation}{section}

\theoremstyle{plain}
\newtheorem{thm}{Theorem}[section]

\newtheorem{prop}[thm]{Proposition}
\newtheorem{coroll}[thm]{Corollary}
\newtheorem{claim}[thm]{Claim}

\theoremstyle{definition}
\newtheorem{defn}[thm]{Definition}

\newtheorem{remark}[thm]{Remark}
\newtheorem{ex}[thm]{Example}

\usepackage{xspace} 

\newcommand{\R}{\mathbb{R}}
\newcommand{\Z}{\mathbb{Z}}

\DeclareMathOperator{\conv}{Conv}
 
\DeclareMathOperator{\inter}{int}
\DeclareMathOperator{\Vis}{Vis}
\newcommand{\ehr}{\mathrm{Ehr}}
\newcommand{\spt}{\mathrm{Base}}

\def\cir{{\rm cir}}
\def\tree{{\rm Base}}
\def\tcQ{{\tilde{\mathcal{Q}}}}

\title{Graph minors, Ehrhart theory, and a monotonicity property}

\begin{document}

\author{Tam\'as K\'alm\'an}
\address{Department of Mathematics,
Institute of Science Tokyo,
Japan}
\email{kalman@math.titech.ac.jp}

\author{Lilla T\'othm\'er\'esz}
\address{ELTE Eötvös Loránd University, P\'azm\'any P\'eter s\'et\'any 1/C, Budapest, Hungary}
\email{lilla.tothmeresz@ttk.elte.hu}



\date{}

\begin{abstract}
We study the extended root polytope associated to a directed graph. 
We show that under the operations of deletion and contraction of an edge of the graph, none of the coefficients of the $h^*$-polynomial of 
the associated extended root polytope increase.
We examine cases when the $h^*$-polynomial does not change, for instance
when contracting the edges of a minimal directed join in a digraph whose lattice polytope has the Gorenstein property.
\end{abstract}

\maketitle

\section{Introduction}
\label{sec:intro}

Any finite directed graph $D=(V,E)$ has naturally associated to it a \emph{root polytope} $\mathcal{Q}_D$ \cite{smooth_Fano}, defined as the convex hull
\[\mathcal{Q}_D=\conv\{\,\mathbf{x}_e \mid e\in E\,\}\subset\mathbb R^V,\]
where $\mathbf{x}_e$ denotes the vector with coordinate $1$ for the head of $e$, coordinate $-1$ for the tail of $e$, and all other coordinates $0$.
The \emph{extended root polytope} is 
\[\tilde{\mathcal{Q}}_D=\conv(\{\mathbf{0}\}\cup \{\, \mathbf{x}_e\mid e\in E\,\})\subset\mathbb R^V.\]

These polytopes have been intensively studied, both for their algebraic and combinatorial properties, as well as for their applications in physics \cite{Sym_edge_appearance,smooth_Fano,arithm_symedgepoly,DAli,alex,fix_elsorrend,ChenDavis}. A notable special case is when $D$ is a bidirected graph (that is, $\overrightarrow{uv}$ and $\overrightarrow{vu}$ are present in $E$ with equal multiplicity for each $u,v\in V$), when we have $\mathcal{Q}_D=\tcQ_D$, and it is called the \emph{symmetric edge polytope} (of the underlying undirected graph).
A major theme regarding
root polytopes is their Ehrhart theory, and this is also the topic of the present paper. In particular, we study the $h^*$-polynomials of extended root polytopes. (The $h^*$-polynomial of a lattice polytope has nonnegative integer coefficients, which sum to the normalized volume; see section \ref{sec:ehrhart} for the formal definition.)

The 
$h^*$-polynomial of $\tcQ_D$
reveals much about the structure of $D$. The association $D\mapsto h^*_{\tcQ_D}$ is closely related to the Tutte polynomial. More precisely, for an undirected graph $G$, the specialization $T_G(x,1)$ of the Tutte polynomial (that is, the generating function of internal activity \cite{Tutte} over the set of the graph's spanning trees) can be identified with $h^*_{\tcQ_D}$ for the digraph $D$ obtained from $G$ by subdividing each edge and orienting the two new edges toward the subdividing point \cite{hiperTutte,KP_Ehrhart}. The polynomial $h^*_{\tcQ_D}$ 
also has connections to greedoid polynomials \cite{eulerian_greedoid}.
In an earlier paper \cite{fokszam} we presented a graph theoretic
formula for the degree of $h^*_{\tcQ_D}$.

For simplicity, let us denote $h^*_{\tcQ_D}$ by $h^*_D$ and call it the \emph{interior polynomial} of $D$. 
In this paper we establish two natural monotonicity properties of $h^*_{D}$ that hold in connection with the two basic graph minor operations, deletion and contraction.

\begin{thm}\label{thm:h_of_deletion}  
Let $D$ be a 
directed graph. For any edge $e$ of $D$, we have that $h^*_{{D-e}}$ is coefficientwise smaller than or equal to $h^*_{D}$.
\end{thm}

\begin{thm}\label{thm:h_of_contraction}
    Let $D$ be a 
    directed graph. For any edge $e$ of $D$, we have that $h^*_{{D/e}}$ is coefficientwise smaller than or equal to $h^*_{D}$.
\end{thm}

The first of these is a direct consequence of the following fundamental result.

\begin{thm}[Stanley's monotonicity theorem, \cite{Stanley_monotonicity}]
     If $P$ and $Q$ are lattice polytopes such that $P\subseteq Q$, then $h^*_P$ is coefficientwise smaller than or equal to $h^*_Q$.
\end{thm}

\begin{proof}[Proof of Theorem \ref{thm:h_of_deletion}]
$\tcQ_{D-e}\subseteq \tcQ_D$ by definition. Now the statement follows from Stanley's monotonicity theorem.
\end{proof}

The proof of Theorem \ref{thm:h_of_contraction} is far more involved and it will occupy much of the paper. 
It is easy to see that $\tcQ_{D/e}$ is the projection of $\tcQ_D$ 
along the direction 
of $\mathbf{x}_e$.
When a lattice polytope is projected onto another one of lower dimension, in general the $h^*$-vector does not behave monotonically, not even if the kernel of the projection is spanned by some of the edge directions. For example, the triangle $\Delta$ with vertices $(0,0)$, $(0,-1)$, and $(3,1)$ has $h^*_\Delta(x)=1+x+x^2$; the projection of $\Delta$ along the second axis is the line segment $S$ between $0$ and $3$, where $h^*_S(x)=1+2x$.
Due to the existence of such examples, our proof of Theorem \ref{thm:h_of_contraction} will have to rely on certain special properties of $\tcQ_D$. 
We note that those same features also allow us to generalize
Theorem \ref{thm:h_of_contraction} 
from directed graphs to oriented regular matroids
in a straightforward way, even though we will not elaborate on the details here. Theorem \ref{thm:h_of_deletion} is just as obvious in the matroid context, too.

The idea of the proof is the following: One can dissect $\tcQ_D$ into simplices that correspond to certain spanning 
forests of $D$. We will call a set of spanning forests 
that yields a dissection a \emph{dissecting forest 
set}. 
We will use a formula for $h^*_{D}$ (a straightforward generalization of \cite[Theorem 1.8]{fix_elsorrend}) that gives the $h^*$-polynomial as the generating function of a certain activity statistic over any dissecting forest set; see Theorem \ref{thm:interior_poly_fixed_order} and the formula \eqref{eq:h*D}. With that, the key is to construct related dissections and related activities for $\tcQ_D$ and for $\tcQ_{D/e}$, so that we can compare the formulas for the two polytopes.

As a byproduct of our computation, 
we also obtain that $h^*_D$ is multiplicative over disjoint unions of digraphs, see Proposition \ref{prop:szorzat}. This is less obvious than one might expect. We mention that it can also be derived from Stapledon's weighted Ehrhart theory \cite{stapledon_weighted}, in particular from \cite[Lemma 4]{stapledon_product}
by noting that the weighted and ordinary $h^*$-polynomials coincide for extended root polytopes. 

We also identify the cases when there is equality.

\begin{prop}\label{prop:equality_in_D-e}
Let $D$ be a 
directed graph. We have $h^*_D=h^*_{D-e}$ if and only if $e$ is a loop or bridge of $D$, or there is another edge in $D$ parallel to $e$.
\end{prop}

\begin{prop}\label{prop:equality_in_D/e}
    Let $D$ be a 
    directed graph, and let $e$ be a non-loop edge in it. We have $h^*_D = h^*_{D/e}$ if and only if $\mathbf{x}_e$ is contained by each facet of $\tcQ_D$ that does not contain $\mathbf{0}$.
\end{prop}

The last assumption on $e$ can be rephrased 
in purely graph-theoretical terms, too, but not in a particularly appealing way. 
There is, however, a simple sufficient condition of equality, see Proposition \ref{prop:elegseges}, stating that for any cycle containing $e$, there are at least as many edges along it that point in the opposite cyclic direction to $e$ as there are edges pointing in the same direction. 

In addition to the above, we characterize digraphs whose extended root polytopes have the so called Gorenstein property, and point out that in the Gorenstein case, our results fit nicely with the existing literature of Gorenstein polytopes. In particular, `special simplices' have a nice combinatorial interpretation in our context, and through Proposition \ref{prop:equality_in_D/e},
Theorem \ref{thm:Gorenstein_special_projection_root_polytope}
gives a graph theoretic interpretation for
the theorem \cite{BN_comb_mirror_sym} (in the special case of extended root polytopes) that the projection of a Gorenstein polytope along a special simplex yields a reflexive polytope with the same $h^*$-polynomial. 
More precisely, we find that 
any minimal directed join of a Gorenstein digraph gives rise to a special simplex and the projection along it is equivalent to contracting the edges of the dijoin; with that, $D$ becomes strongly connected and thus $\tcQ_D$ reflexive, while the $h^*$-polynomial does not change.

The structure of the paper is as follows: Section \ref{sec:prelim} contains some definitions and fundamental results about the extended root polytope. In Section \ref{sec:h-vector_formula}, we give a formula for $h^*_D$ in terms of certain activities. 
In Section \ref{sec:cycle_signa_and_triang}, we recall the basics of acyclic circuit signatures, and explain how extended root polytopes can be dissected using these signatures. Section \ref{sec:proof} puts together these ingredients and proves Theorem \ref{thm:h_of_contraction}, as well as the multiplicativity of the interior polynomial over disjoint unions. In Section \ref{sec:equality}, we characterize the cases 
of equality and mention some classes of examples when it holds. Section \ref{sec:gorenstein} discusses digraphs with the Gorenstein property.

\section{Preliminaries}\label{sec:prelim}

\subsection{Graph notations}

A directed graph (digraph for short) is \emph{weakly connected} if the undirected graph obtained by forgetting the orientations is connected.

A digraph is \emph{strongly connected} if there is a directed path from $u$ to $v$, as well as from $v$ to $u$, for any pair of vertices $u$ and $v$.

A \emph{cut} of a digraph is a non-empty set of edges $C^*$ so that there is a partition $V_0\sqcup V_1$ of the vertices such that $C^*$ contains exactly the edges going between $V_0$ and $V_1$. In this case we call $V_0$ and $V_1$ the \emph{shores} of the cut. A cut is \emph{elementary} if it is minimal with respect to inclusion among cuts, that is, if its 
removal increases the number of connected components by exactly $1$.
The cut is called \emph{directed} if either each edge points from $V_0$ to $V_1$ or each edge points from $V_1$ to $V_0$.

An edge that forms a one-element cut is called a \emph{bridge}.
A set of edges is called a \emph{directed join}, or \emph{dijoin} for short, if it contains at least one edge from each directed cut.

A \emph{spanning tree} $T$ of a digraph $D$ is a subgraph whose underlying undirected graph is a tree that contains all vertices. (In particular, the orientation does not play a role in the definition.) Only connected graphs have spanning trees; in general, we consider subgraphs consisting of one spanning tree from each connected component, and call these \emph{spanning forests}. We will identify spanning trees and spanning forests with their edge sets, e.g., write $e\in F$ if the edge $e$ is in the spanning forest $F$. We denote the set of spanning forests of the digraph $D$ by $\spt(D)$.

\subsection{Definition of the $h^*$-polynomial}
\label{sec:ehrhart}

Let $Q\subset\mathbb{R}^n$ be a $d$-dimensional lattice polytope (that is, its vertices are in $\mathbb{Z}^n$).
The \emph{$h^*$-polynomial} (also commonly called the \emph{$h^*$-vector}) of $Q$ is the polynomial $\sum_{i=0}^d h^*_i t^i$ defined by Ehr\-hart's identity 
\begin{equation}
\label{eq:h-csillag}
\sum_{i=0}^d h^*_i t^i = (1-t)^{d+1} \ehr_Q(t),
\quad\text{where}\quad 
\ehr_Q(t)=\sum_{k=0}^\infty|(k\cdot Q)\cap\Z^n|\,t^k
\end{equation}
is the so called \emph{Ehrhart series} of $Q$. We note that $h^*_0=1$ whenever $d\ge0$, i.e., whenever $Q$ is non-empty.

Intuitively, the $h^*$-polynomial can be thought of as a refinement of volume. Indeed, $h^*(1)$ (that is, the sum of the coefficients) is equal to the normalized volume of the polytope, where by normalized we mean that the volume of a  $d$-dimensional unimodular simplex is $1$. For a more detailed introduction to $h^*$-polynomials, see \cite[Chapters 3 and 10]{BeckRobbins}.

\subsection{The extended root polytope}

Here, we gather results on the dimension and facets of extended root polytopes. Before citing these, let us remark that 
we may assume that $D$ is loopless and does not have any parallel edges. 

\begin{claim}\label{cl:loop_and_parallel_edge}
If $e$ is a loop edge, or 
if there is at least one more edge 
with the same initial and terminal points as $e$,
then $\tcQ_D=\tcQ_{D-e}$.    
\end{claim}

\begin{proof}
If $e$ is a loop then $\mathbf{x}_e=\mathbf{0}$.
Hence if $e$ is a loop or if $e$ has parallel copies, then $\tcQ_{D-e}$ has the same generators as $\tcQ_{D}$.
\end{proof}

\begin{prop}\cite{semibalanced,numata}\label{prop:root_poly_dim}
    For a weakly connected digraph $D$, we have $\dim(\tcQ_D)=|V|-1$. More generally $\dim(\tcQ_D)=|V(D)|-c(D)$, where $c(D)$ is the number of weakly connected components of $D$.
\end{prop}

For the description of the facets, we need some additional notions. 
First of all, we will identify our vector spaces and their duals by using the standard dot product. 
To a cut $C^*$ with shores $V_0$ and $V_1$, we associate a functional/vector $f_{C^*}$ by
defining $f_{C^*}(v)=1$ for $v\in V_1$ and $f_{C^*}(v)=0$ for $v\in V_0$. Note that if $C^*$ is directed with each edge pointing from $V_0$ to $V_1$, then $f_{C^*}\cdot \mathbf{x}_e=1$ if $e\in C^*$, and $f_{C^*}\cdot \mathbf{x}_e=0$ if $e\notin C^*$.

An \emph{admissible layering} is a function $\ell\colon V\to\Z$, or in other words, a vector $\ell\in\Z^V$, so that $\ell\cdot\mathbf x_e\le1$ for all edges $e$ of $D$, and the edges $e$ with $\ell\cdot\mathbf x_e=1$ (which we will sometimes call the \emph{tight} edges with respect to $\ell$) form a spanning
subgraph (i.e., include a spanning forest) of $D$. We call two 
admissible layerings 
\emph{equivalent} if they differ by some 
function that is constant on every connected component of $D$.

It is easy to see that the extended root polytope of a digraph is the so called free sum of those of its connected components. Here if $\mathbf0\in R\subset U$ and $\mathbf0\in S\subset W$ are convex sets in the real vector spaces $U$ and $W$, respectively, then their \emph{free sum} is 
\[R\oplus S=\conv\left((R\times\{\mathbf0\})\cup(\{\mathbf0\}\times S)\right)\subset U\oplus W.\]
In an earlier paper we gave the facet description of the extended root polytope in the weakly connected case \cite[Corollary 2.11]{fokszam}. From that we readily obtain the following.

\begin{prop}
	\label{prop:facets}
	For any digraph $D$, the facets of $\tcQ_D$ are as follows: 
	\begin{enumerate}
		\item \label{partone}
		$C^*\mapsto\{\mathbf x\in\tcQ_D\mid f_{C^*}\cdot\mathbf x=0\}$ gives a one-to-one correspondence between the elementary directed cuts of $D$ and facets of $\tcQ_D$ containing $\mathbf0$. 
        		\item \label{parttwo}
		$\ell\mapsto\{\mathbf x\in\tcQ_D\mid\ell\cdot\mathbf x=1\}$ induces a bijection between the equivalence classes of admissible layerings and facets of $\tcQ_D$ that do not contain $\mathbf0$.
	\end{enumerate}
\end{prop}

\section{A formula for the $h^*$-polynomial}
\label{sec:h-vector_formula}

To prove Theorem \ref{thm:h_of_contraction}, we need a formula for the $h^*$-polynomial of the extended root polytope. For this we will rely on an extension of \cite[Theorem 1.8]{fix_elsorrend} from bidirected graphs to arbitrary directed graphs. 

To compute the $h^*$-polynomial of a lattice polytope it is very useful to dissect it into simplices, especially unimodular ones. As to where to find such simplices in the case of polytopes derived from graphs, the following fact is simple and well known in the connected case. A proof can be found, for example, in \cite[Lemma 3.5]{semibalanced}. The general case follows because free sums of unimodular simplices (containing the origin) are again unimodular simplices.

\begin{prop}
Let $D$ be a 
digraph and let $F$ be a spanning forest of $D$. Then $\tcQ_F$ is a unimodular simplex of maximal dimension within $\tcQ_D$.
\end{prop}

\begin{defn}
    We call a set of spanning forests $\mathcal{F}$ of $D$ a \emph{dissecting forest set} of $D$ if the simplices in $\{\tcQ_F\mid F\in \mathcal{F}\}$ form a dissection of $\tcQ_D$, that is, they are interior disjoint, and their union is $\tcQ_D$.
\end{defn}

We will give a formula that describes the $h^*$-polynomial of the extended root polytope as the generating function of a certain passivity statistic over a dissecting forest set. (We will generally use `passive' as the negation of the more common term `active.') To define these passivities, we need some additional definitions.

For a spanning forest $F$ of $D$ and an edge $e\in F$, let $T$ denote the connected component of $F$ that contains $e$. Then, the \emph{fundamental cut} of $e$ with respect to $F$, denoted by $C^*(F,e)$, is the set of edges of $D$ (including $e$) that connect the two components of $T-e$. 
We write $C^*_D(F,e)$ if we want to emphasize the underlying digraph $D$.
We say that an edge $e'\in C^*(F,e)$
\emph{stands parallel} to $e$ if the heads of $e$ and $e'$ are in the same component of $T-e$. 
Otherwise we say that $e'$ \emph{stands opposite} to $e$. 

For a spanning forest $F$ and an edge $e\notin F$, the \emph{fundamental cycle} of $e$ with respect to $F$, denoted by $C(F,e)$, is the unique cycle in the subgraph $F\cup e$.
We say that an edge $e'\in C(F,e)$
\emph{stands parallel} to $e$ if they point in the same cyclic direction within $C(F,e)$. Otherwise we say that $e'$ \emph{stands opposite} to $e$. (The same terms apply to any cycle containing both $e$ and $e'$.)

\begin{defn}[internal semi-activity in digraphs \cite{hyperBernardi}]
\label{def:semiactive}
	Let $D$ be a digraph with a fixed ordering of the edges. Let $F$ be a spanning forest of $D$. An 
	edge $e\in F$ is \emph{internally semi-active} for $F$ if in the fundamental cut $C^*(F,e)$, the minimal edge (with respect to the fixed ordering) stands parallel to $e$. If the minimal edge stands opposite to $e$, then we say that $e$ is \emph{internally semi-passive} for $F$. 
	
	The \emph{internal semi-activity} of a spanning forest (with respect to the fixed order) is the number of its internally semi-active edges, while the \emph{internal semi-passivity} is the number of internally semi-passive edges.
\end{defn}

This notion of activity is the dual pair of 
``external semi-activity'' \cite{LiPostnikov}.
Internal semi-activity is similar to Tutte's concept of internal activity \cite{Tutte}, but instead of requiring $e\in F$ to be the minimal element in $C^*(F,e)$, it only requires $e$ to stand parallel to the minimal edge of $C^*(F,e)$.

\begin{thm}\label{thm:interior_poly_fixed_order}
	Let $\mathcal{F}$ be any dissecting forest set for the digraph
	$D$. Fix an ordering of the 
 edges of $D$.
	Then the $h^*$-vector of the extended root polytope
 $\tcQ_D$ satisfies
	$$\left(h^*_{D}\right)_i = |\{F\in \mathcal{F}\mid F \text{ has exactly $i$ internally semi-passive edges}\}|.$$
\end{thm}

This theorem can be proved the same way as Theorem 1.8 in \cite{fix_elsorrend}.
For the sake of completeness, we include the proof in the appendix.

\begin{ex}\label{ex}
The digraph of Figure \ref{fig:triang} has $h^*(x)=1+x$. To see this,
notice that the trees of the second and third panels form a dissecting forest set. (The extended root polytope can be seen in the fourth panel.) For the ordering $e_1<e_2<e_3$, the tree in the second panel has $0$ internally semi-passive edges, while the tree in the third panel has $1$ internally semi-passive edge, namely $e_3$.
The dissecting forest set is in fact unique in this case. In particular, the spanning forest $\{e_1,e_3\}$ cannot be extended to a dissecting forest set.
\end{ex}

\begin{figure}
	\begin{tikzpicture}
	\begin{scope}[shift={(-4.7,0)},scale=.80]
	\node [] (0) at (0,0.8) {\small $e_1$};
	\node [] (0) at (1,-0.3) {\small $e_2$};
	\node [] (0) at (2,0.8) {\small $e_3$};
	\node [circle,scale=.6,draw] (1) at (-0.5,0) {$u$};
	\node [circle,scale=.6,draw] (2) at (1,1) {$v$};
	\node [circle,scale=.6,draw] (4) at (2.5,0) {$w$};
	\path [thick,->,>=stealth] (1) edge [above] node {} (2);
	\path [thick,->,>=stealth] (1) edge [below] node {} (4);
	\path [thick,->,>=stealth] (2) edge [above] node {} (4);
	\end{scope}
 \begin{scope}[shift={(-1.4,0)},scale=.60]
	\node [] (0) at (1,-0.5) {1};
	\node [] (0) at (1.5,1.2) {};
	\node [] (0) at (2.9,0) {};
	\node [circle,fill,scale=.6,draw] (1) at (-0.5,0) {};
	\node [circle,fill,scale=.6,draw] (2) at (1,1) {};
	\node [circle,fill,scale=.6,draw] (4) at (2.5,0) {};
	\path [thick,->,>=stealth] (1) edge [left] node {} (2);
	\path [thick,->,>=stealth] (1) edge [above] node {} (4);
	\path [dashed,->,>=stealth] (2) edge [above] node {} (4);
	\end{scope}
 \begin{scope}[shift={(1.4,0)},scale=.60]
	\node [] (0) at (1,-0.5) {$x$};
	\node [] (0) at (1.5,1.2) {};
	\node [] (0) at (2.9,0) {};
	\node [circle,fill,scale=.6,draw] (1) at (-0.5,0) {};
	\node [circle,fill,scale=.6,draw] (2) at (1,1) {};
	\node [circle,fill,scale=.6,draw] (4) at (2.5,0) {};
	\path [dashed,->,>=stealth] (1) edge [left] node {} (2);
	\path [thick,->,>=stealth] (1) edge [above] node {} (4);
	\path [thick,->,>=stealth] (2) edge [above] node {} (4);
	\end{scope}
 \begin{scope}[shift={(4.5,0.5)},scale=.5]
        \draw[color=gray,fill=gray] (-0.5,0) -- (1,1) -- (2.5, 0) -- (1,-1) -- cycle;
	\node [] (0) at (1.5,-1.3) {$\mathbf{0}$};
	\node [] (0) at (-0.6,-0.5) {$\mathbf{x}_{e_1}$};
	\node [] (0) at (2.9,-0.6) {$\mathbf{x}_{e_3}$};
	\node [] (0) at (1.6,1.3) {$\mathbf{x}_{e_2}$};
	\node [circle,fill,scale=.4,draw] (1) at (-0.5,0) {};
	\node [circle,fill,scale=.4,draw] (2) at (1,1) {};
	\node [circle,fill,scale=.4,draw] (3) at (1,-1) {};
	\node [circle,fill,scale=.4,draw] (4) at (2.5,0) {};
        \path [thick,-] (1) edge [left] node {} (2);
	\path [thick,-] (1) edge [above] node {} (3);
	\path [thick,-] (2) edge [above] node {} (4);
	\path [thick,-] (3) edge [above] node {} (4);
	\end{scope}
        \end{tikzpicture}
\caption{Illustration for Example \ref{ex}.}\label{fig:triang}
\end{figure}

\section{Circuit signatures and dissections}
\label{sec:cycle_signa_and_triang}

In this section we discuss a general method of finding dissecting forest sets, to which later we can apply Theorem \ref{thm:interior_poly_fixed_order}. In fact, we will end up with sets of forests inducing regular triangulations, even though we do not need this stronger property and we will not explicitly prove it. The construction relies on the following notion \cite{BBY}.

\subsection{Acyclic circuit signatures}
Let $D$ be a digraph, and let $C$ be a cycle in $D$. A \emph{signed cycle} is an ordered partition $\overrightarrow{C}=C^+\sqcup C^-$ so that $C^+$ contains the edges of $C$ going in one of the cyclic directions, and $C^-$ contains the edges of $C$ going in the other cyclic direction. Naturally, each cycle supports two signed cycles that can be obtained from one another by switching the roles of $C^+$ and $C^-$. We call $C^+$ and $C^-$ the two \emph{arcs} of $\overrightarrow{C}$. The \emph{vector} of a signed cycle $\overrightarrow{C}$, denoted by $\chi_{\overrightarrow{C}}\in\mathbb{Z}^E$, has coordinate $+1$ corresponding to $e$ if $e\in C^+$, has $-1$ if $e\in C^-$, and $0$ if $e\notin C$.

A \emph{circuit signature} $\sigma$ is a collection of signed cycles such that for each cycle $C$, 
exactly one of the signed cycles supported on $C$ is contained in $\sigma$. By a slight abuse of notation, we denote by $\sigma(C)$ the signed cycle in $\sigma$ supported on $C$.
A circuit signature $\sigma$ is called \emph{acyclic} \cite{BBY} if for any non-empty set of cycles $C_1, \dots, C_s$ and 
positive coefficients $a_1, \dots, a_s$, we have $\sum_{i=1}^s a_i\cdot \chi_{\sigma(C_i)}\neq \mathbf{0}$. 

Let $w\colon E\to\mathbb{R}$ be a function, which in this context we will call a \emph{weight function}. 
We say that $w$ is \emph{generic} if we have $\sum_{f\in C^+} w(f) \neq \sum_{f\in C^-} w(f)$ for each signed cycle of $D$. 
This gives rise to a circuit signature 
in the following way.

\begin{defn}\label{def:signature_from_weight}
For a generic weight function $w\colon E\to \mathbb{R}$, let the \emph{induced circuit signature} $\cir^w$ be the one consisting of those signed cycles $\overrightarrow{C}=C^+\sqcup C^-$ that satisfy $\sum_{f\in C^+} w(f) > \sum_{f\in C^-} w(f)$.
\end{defn}

It is easy to see that $\cir^w$ is acyclic, since the vector $\sum_{i=1}^s a_i \chi_{\sigma(C_i)}$, with each $a_i > 0$, always has a positive scalar product with $w\in \mathbb{R}^E$. In fact, it is a consequence of Farkas' lemma that all acyclic circuit signatures arise this way.

\begin{prop}\cite[Lemma 2.3.1]{BBY}\label{prop:acycl_signature_char}
A signature $\sigma$ is acyclic if and only if $\sigma=\cir^w$ for some generic weight function $w$.
\end{prop}

An acyclic circuit signature $\sigma$ of $D$ also induces an acyclic circuit signature $\sigma/e$ for $D/e$ in a natural way \cite{consistency_matroids}, where $e$ is an arbitrary edge of the digraph $D$. Let us review this construction.

Recall what the cycles of $D/e$ look like: If $e\in C$ for a cycle $C$ of $D$, then $C-e$ is a cycle of $D/e$. If $e\notin C$ for a cycle $C$, then it might happen that $C$ is still a cycle in $D/e$ (if the cycle does not contain both endpoints of $e$), or it might be that  $C$ becomes two cycles $C_1$ and $C_2$ glued at a vertex (if both endpoints of $e$ are along $C$). However, in this latter case $C_1\cup e$ and $C_2\cup e$ are also cycles of $D$, and we get $C_1$ and $C_2$ from them by the first method. Hence we can say that the cycles of $D/e$ are either also cycles in $D$, or of the form $C-e$ where $C\ni e$ is a cycle of $D$.

Knowing this, it is quite natural to define the circuit signature $\sigma/e$: If $e\in C$ for a cycle $C$ of $D$, then let $(\sigma/e)(C-e)=\sigma(C)|_{E-e}$, by which we mean that $((\sigma/e)(C-e))^+=(\sigma(C))^+\cap (E-e)$ and $((\sigma/e)(C-e))^-=(\sigma(C))^-\cap (E-e)$. If $e\notin C$ where $C$ is a cycle of both $D$ and $D/e$, then let $(\sigma/e)(C)=\sigma(C)$.
By the above remark, this way we have assigned exactly one signed cycle to each underlying cycle.

The following statement was proved in \cite{consistency_matroids}, but we repeat the proof since it is very short.

\begin{prop}\cite[Lemma 5.8]{consistency_matroids}\label{prop:acyclic_sign_contraction}
    If $\sigma$ is an acyclic circuit signature of $D$, and $e$ is a non-loop edge, then $\sigma/e$ is an acyclic circuit signature of $D/e$.
\end{prop}

\begin{proof}
Take an arbitrary nonnegative linear combination $\sum_{i=1}^s a_i \chi_{(\sigma/e)(C_i)}$ of vectors of signed circuits, summing to zero. For each $i$
there is a unique cycle $C'_i$ of $D$
such that $\chi_{\sigma(C'_i)}|_{E - e} = \chi_{(\sigma/e)(C_i)}$.
Now consider $\sum_{i=1}^s a_i \chi_{\sigma(C'_i)}$, which must be everywhere zero except possibly for its $e$-coordinate.
But as this sum, interpreted as a system of non-negative values associated to the edges of the digraph, 
has the property that the in-flow at each vertex equals the out-flow,
our assumption that $e$ is not a loop implies that the value on $e$ is also $0$. This in turn implies, by the acyclicity of $\sigma$, that the $a_i$'s are zeros.
\end{proof}

\subsection{Dissections via acyclic circuit signatures}

As we have mentioned earlier, acyclic circuit signatures can be used to construct dissections of extended root polytopes. This is explained in \cite{LiPostnikov} 
in the `flat' case, when the generating vectors of the root polytope lie in an affine hyperplane (not containing $\mathbf0$). Let us give here a construction for the general case. We choose to give the proofs, although we note that one could also reduce the general case to the flat one.

Recall that we denote the set of spanning forests of a digraph $D$ by $\spt(D)$.

\begin{defn}
\label{def:dissecting_forests}
We say that the spanning forest $F$ is \emph{compatible} with the circuit signature $\sigma$ if $e\in \sigma(C(F,e))^+$ for each edge $e\in E-F$.

We denote the set of spanning forests of $D$ compatible with $\sigma$ by $\tree(D,\sigma)$.
\end{defn}

It follows from Li and Postnikov's results \cite{LiPostnikov} 
that if $D$ is a digraph where the two arcs of each cycle have equal cardinality, and $\sigma$ is an acyclic circuit signature, then $\tree(D,\sigma)$ is a dissecting forest set of $D$. (We often call these graphs \emph{semi-balanced} and they correspond to the flat case mentioned above.) For general graphs, one needs to be more careful, as the following example shows. 

\begin{ex}
    Consider the digraph $D$ of Figure \ref{fig:triang}, and take $\sigma=\cir^w$ with $w(e_1)=w(e_3)=1$ and $w(e_2)=3$. Then, for the unique cycle in the graph, the positive arc is $\{e_2\}$ and the negative arc is $\{e_1,e_3\}$. Hence $\tree(D,\sigma)$ consists of the unique tree with edges $e_1$ and $e_3$. However, as we can see in the fourth panel of Figure \ref{fig:triang}, $\tcQ_T$ for $T=\{e_1,e_3\}$ is only a proper subset of $\tcQ_D$.
\end{ex}

However, it is still possible to define dissecting forest sets based on acyclic circuit signatures, provided that one adds one more condition.
Let us call a circuit signature $\sigma$ \emph{long arc positive}, if for each cycle $C$ we have $|\sigma(C)^+| \geq |\sigma(C)^-|$. (Note that long arc positivity is automatically satisfied if each cycle has two equal arcs.)

\begin{prop}\label{prop:triang_tree_set_from_signature}
Let $D$ be a directed graph.
If $\sigma$ is an acyclic, long arc positive circuit signature, then $\tree(D,\sigma)$ is a dissecting forest set of $D$.
\end{prop}

\begin{proof}
We first show that for $F_1, F_2 \in \tree(D,\sigma)$, the simplices $\tcQ_{F_1}$ and $\tcQ_{F_2}$ are (relative) interior disjoint. We note that this property holds even if $\sigma$ is only acyclic and not necessarily long arc positive.

Suppose for a contradiction that there is a point $\mathbf{p}\in \inter(\tcQ_{F_1})\cap \inter(\tcQ_{F_2})$. Then, $\mathbf{p}=\sum_{e\in F_1} \lambda_e \mathbf{x}_e =\sum_{e\in F_2} \mu_e \mathbf{x}_e$ with $\sum_{e\in F_1}\lambda_e < 1$ and $\sum_{e\in F_2}\mu_e < 1$ (the sums are smaller than $1$ because $\mathbf{0}$ is a vertex in both simplices, and it has to have a positive coefficient), moreover $\lambda_e > 0$ for each $e\in F_1$ and $\mu_e > 0$ for each $e\in F_2$. Define $\lambda_e=0$ for $e\notin F_1$ and $\mu_e = 0$ for $e\notin F_2$. As $F_1\neq F_2$, there exists some $e\in E$ with $\lambda_e \neq \mu_e$, whence we have a nontrivial linear relation $\mathbf{0}=\sum_{e\in E}(\lambda_e - \mu_e)\mathbf{x}_e$.

In this case there exists a signed circuit $\overrightarrow{C}$ such that $C^+ \subseteq \{e\in E\mid \lambda_e-\mu_e > 0\}$ and $C^- \subseteq \{e\in E\mid \lambda_e-\mu_e < 0\}$. (This is proved for example in \cite[Claim 2.9]{eulerian_greedoid}.) In particular, $C\subseteq F_1\cup F_2$. Also, for $e\in F_1-F_2$ we have $\lambda_e-\mu_e=\lambda_e > 0$, in other words $(F_1-F_2)\cap C \subseteq C^+$, and for $e\in F_2-F_1$ we have $\lambda_e-\mu_e=-\mu_e < 0$, whence $(F_2-F_1)\cap C \subseteq C^-$.

On the other hand, we can write $\chi_{\overrightarrow{C}}$ as a sum of vectors of signed fundamental cycles of $F_1$, as well as of $F_2$. Indeed $-\chi_{\overrightarrow{C}}=\sum_{e\in C-F_1} \chi_{\overrightarrow{C}(F_1,e)}$, which is true since each $e\in C-F_1 \subset (F_2-F_1)\cap C$ is in $C^-$. Similarly $\chi_{\overrightarrow{C}}=\sum_{e\in C-F_2} \chi_{\overrightarrow{C}(F_2,e)}$, which is true since each $e\in C-F_2 \subset (F_1-F_2)\cap C$ is in $C^+$. As $F_1, F_2\in \tree(D,\sigma)$, the signed fundamental cycles $\overrightarrow{C}(F_i, e)$ for $i=1,2$ and $e\notin F_i$, are in $\sigma$. Thus $\mathbf0=\sum_{e\in C-F_1} \chi_{\overrightarrow{C}(F_1,e)} + \sum_{e\in C-F_2} \chi_{\overrightarrow{C}(F_2,e)}$ is a positive linear combination of signed circuits in $\sigma$, contradicting the assumption that $\sigma$ is acyclic.

Now it is enough to show that $\tcQ_D\subseteq \bigcup\{\tcQ_F \mid F\in \tree(D,\sigma)\}$.

Let $\mathbf{p}$ be an arbitrary point in $\tcQ_D$. By Caratheodory's theorem, we can choose $|V|-c(D)+1$ affine independent generators such that $\mathbf{p}$ is in their convex hull. If $\mathbf0$ is one of these vectors, that means that we have $\mathbf{p}\in \tcQ_F$ for some spanning forest $F$ (not necessarily in $\tree(D,\sigma)$). 

Else if $\mathbf0$ is not one of the vectors, then we have $\mathbf{p}\in \mathcal{Q}_S$ for some $S$ where $S\subset D$ contains a cycle $C$. Let us express this as $\mathbf{p}=\sum_{e\in S}\lambda_e \mathbf{x}_e$, where each $\lambda_e\geq0$ and $\sum_{e\in S}\lambda_e=1$. 
This implies 
\[\mathbf{p}=\sum_{e\in S-C}\lambda_e \mathbf{x}_e+\sum_{e\in \sigma(C)^+}(\lambda_e - \varepsilon) \mathbf{x}_e + \sum_{e\in \sigma(C)^-}(\lambda_e + \varepsilon) \mathbf{x}_e,\] 
where $\varepsilon=\min\{\lambda_e \mid e\in \sigma(C)^+\}$. The right hand side is a nonnegative linear combination, in which the coefficient of one edge of $C$ became $0$, and where the sum of the coefficients is $1-(|\sigma(C)^+|-|\sigma(C)^-|)\varepsilon \leq 1$ by the long arc positivity assumption. 
Therefore we can add $\mathbf0$ with a nonnegative coefficient to obtain a convex combination. We have thus found, just like in the previous case, a spanning forest $F\subset S$ such that $\mathbf{p}\in \tcQ_F$.

Now we have to deal with the possibility that $F\notin \tree(D,\sigma)$. That means that for some $e_0\notin F$ we have $e_0\in \sigma(C(F,e_0))^-$. Let again $\mathbf{p}=\sum_{e\in F}\lambda_e \mathbf{x}_e + \lambda_0 \cdot \mathbf{0}$.
We employ the same trick as before, that is, let $\lambda_{e_0} = 0$, and re-write our convex combination for $\mathbf p$ as
\begin{multline*}
\mathbf{p}= \sum_{e\in F-C(F,e_0)} \lambda_e \mathbf{x}_e + \sum_{e\in \sigma(C(F,e_0))^+}(\lambda_e-\varepsilon)\mathbf{x}_e + \sum_{e\in \sigma(C(F,e_0))^-}(\lambda_e+\varepsilon)\mathbf{x}_e\\ +\left(\lambda_0 + (|\sigma(C(F,e_0))^+|-|\sigma(C(F,e_0))^-|)\varepsilon\right)\cdot\mathbf{0},
\end{multline*}
where $\varepsilon =\min\{\lambda_e \mid e\in \sigma(C(F,e_0))^+ \}$. Since $\sigma$ is long arc positive, we have $|\sigma(C(F,e_0))^+|-|\sigma(C(F,e_0))^-|\geq 0$; furthermore, as $e_0\in \sigma(C(F,e_0))^-$ by assumption, the coefficient  $\lambda_{e_0} = 0$ has increased to $\varepsilon$ (or stayed the same, if $\varepsilon=0$).
Therefore the new expression is again a convex combination.
As the coefficient of some $g\in \sigma(C(F,e_0))^+$ is $0$ in the new convex combination, we may take $F'=F-g+e_0$ as another spanning forest of $D$ so that the associated simplex contains $\mathbf p$. Now let us show that in a well-defined sense, we have improved our situation.

Since $\sigma$ is acyclic, by Proposition \ref{prop:acycl_signature_char}, there exists a weight function $w\colon E\to\mathbb R$ such that $\sigma= \cir^w$. Let us consider the number $\mathrm{value}(F,\mathbf p)=\sum_{e\in E} \lambda_e w(e)$ associated to any spanning forest $F$ of $D$, where the $\lambda_e$ are the barycentric coordinates of $\mathbf p$ with respect to $\tcQ_F$ (except for the one corresponding to $\mathbf0$), extended as $0$ to the non-edges of $F$. Then for the two forests $F$ and $F'$ above, we have
\begin{align*}
\mathrm{value}(F',\mathbf p) & =  \mathrm{value}(F,\mathbf p)-\varepsilon \cdot \left(\sum_{e\in\sigma(C(F,e_0))^+}w(e)- \sum_{e\in\sigma(C(F,e_0))^-}w(e)\right) \\ 
& < \mathrm{value}(F,\mathbf p),
\end{align*}
where we used the
the definition of $\cir^w=\sigma$.
This means that if there is any $e_0\notin F$ such that $e_0\in \sigma(C(F,e_0))^-$, then we can find another forest $F'$ such that $\mathbf{p}\in\tcQ_{F'}$ and $\mathrm{value}(F',\mathbf p)<\mathrm{value}(F,\mathbf p)$. As there are finitely many spanning forests, we cannot continue this indefinitely, which means that there is a forest $F$ such that $F\in \tree(D,\sigma)$ and $\mathbf{p}\in\tcQ_F$.
\end{proof}

We note again that one could also deduce the 
previous result from the flat case (discussed in \cite{LiPostnikov}) by embedding each generator of the extended root polytope to one higher dimension, with last coordinate equal to $1$, and setting the weight of $\mathbf0$ to be a negative number with absolute value an order of magnitude larger than the weights of the other vectors.

\section{Proof of Theorem \ref{thm:h_of_contraction} and a product formula}
\label{sec:proof}

Now we are ready to prove our main theorem. 

\begin{proof}[Proof of Theorem \ref{thm:h_of_contraction}]
By Claim \ref{cl:loop_and_parallel_edge} we may suppose that $D$ does not have any loops and parallel edges.    
We will use Theorem \ref{thm:interior_poly_fixed_order} and the dissecting forest sets discussed in Section \ref{sec:cycle_signa_and_triang} to compare the $h^*$-polynomials of $\tcQ_D$ and $\tcQ_{D/e}$.

First, let us derive a formula for $h^*_D$.
Theorem \ref{thm:interior_poly_fixed_order} requires that we specify an ordering of the edges of $D$, as well as that we fix a dissecting forest set for $D$.

Let us choose an ordering of the edges of $D$ such that $e$, the edge to be contracted, is the minimal element. Let $\pi\colon E\to \{1, \dots , |E|\}$ express the position of each edge in the ordering, that is, if $e < e_2 < \dots 
< e_{|E|}$ then $\pi(e)=1$ and for $i\geq 2$ we have $\pi(e_i)=i$.

As to the dissection (which, in fact, will be a triangulation) of $\tcQ_D$, we will also rely on the ordering $\pi$ in our construction. More precisely, we define the weight function 
\begin{equation}
\label{eq:sulyfuggveny}
w(f)=1- 2^{-\pi(f)-1}
\end{equation}
for $f\in E$.
Notice that $w$ is generic, moreover that for any subset $S\subseteq E$, we have $|S|-\frac{1}{2} < \sum_{f\in S} w(f) \leq |S|$.
Take $\sigma=\cir^w$ as in Definition \ref{def:signature_from_weight}, and let $\mathcal{F}=\tree(D,\sigma)$, cf. Definition \ref{def:dissecting_forests}.

\begin{claim}
    $\sigma$ is an acyclic, long arc positive circuit signature.
\end{claim}

\begin{proof}
    Acyclicity follows from Proposition \ref{prop:acycl_signature_char}. For long arc positivity, notice that if we have a signed cycle $\overrightarrow{C}$ with $|C^-| > |C^+|$, then 
    \begin{multline*}
    \sum_{f\in C^-}w(f)=|C^-|-\left(\sum_{f\in C^-} 2^{-\pi(f)-1}\right)>|C^-|-\frac12\\ 
    >|C^+|\geq |C^+|- \left(\sum_{f\in C^+} 2^{-\pi(f)-1}\right)=\sum_{f\in C^+}w(f),\end{multline*}
    whence $\overrightarrow{C}$ is not in $\cir^w$.
\end{proof}

Thus, by Proposition \ref{prop:triang_tree_set_from_signature}, the collection $\mathcal F$ is a dissecting forest set for $D$.
Then by Theorem \ref{thm:interior_poly_fixed_order}, we have
\begin{equation}
    \label{eq:h*D}
h^*_{D}(x)=\sum_{F\in \mathcal F} x^{p_{\pi}(F)},
\end{equation}
where $p_{\pi}(F)$ is the internal semipassivity of $F$ with respect to $\pi$, cf.\ Definition \ref{def:semiactive}.

Now let us turn to $D/e$. 
As our edge ordering, we will use the restriction $\pi'$ of $\pi$ to $E(D/e)=E-e$. 
Our dissecting forest set will be closely related to $\mathcal F$. Namely, we let 
\[\mathcal{F}'=\{F' \in \spt(D/e) \mid F'\cup e\in \mathcal{F}\}.\]
It is clear that the elements of $\mathcal F'$ are spanning forests of $D/e$, but we still have to ascertain that they do form a dissecting forest set for $D/e$. For this, we will show that on the one hand, $\mathcal F'$ arises as $\mathcal F'=\tree(D/e, \sigma/e)$, and on the other hand, that $\sigma/e$ is not only acyclic (by Proposition \ref{prop:acyclic_sign_contraction}), but also a long arc positive circuit signature of $D/e$.

The fact that $\mathcal F'=\tree(D/e, \sigma/e)$ is quite easy to see. Indeed, the set of non-edges of any forest $F'\in\mathcal F'$ and the corresponding forest $F=F'\cup e\in\mathcal F$ are the same. For each such edge $f$, its two respective fundamental cycles are either the same or they differ by only the edge $e$. By the definition of the circuit signature $\sigma/e$, the edge $f$ belongs to the positive arc of one cycle if and only if the same is true for the other cycle.

As to the long arc positivity of $\sigma/e$, the only way it could fail is if some cycle $C$ in $D$ had two arcs of equal size and $e$ was part of $\sigma(C)^+$. But because $\sigma=\cir^w$ and $\pi(e)=1$ (cf.\ \eqref{eq:sulyfuggveny}),
this is impossible.
Therefore $\mathcal F'$ is indeed a dissecting forest set.

Now let us apply Theorem \ref{thm:interior_poly_fixed_order} to $D/e$, the forest set $\mathcal F'$, and the ordering $\pi'$ induced by $\pi$ on $E-e$.
It tells us that 
$$h^*_{D/e}(x)=\sum_{F'\in \mathcal F'} x^{p_{\pi'}(F')},$$ 
where $p_{\pi'}(F')$ is the internal semipassivity of $F'$ with respect to $\pi'$. As for each $F'\in \mathcal F'$ we have
$F=F'\cup e\in \mathcal F$, and the correspondence $F'\mapsto F$ is one-to-one, it is enough to show that for such pairs $p_\pi(F)=p_{\pi'}(F')$ holds. 

We will prove this by showing that each edge $f\in F'$ is such that $f$ is internally semi-passive with respect to $F'$ if and only if it is internally semi-passive in $F$. Moreover, we claim that $e$ is not internally semi-passive in $F$. 

Because we obtained $F'$ by contracting $e\in F$, for each $f\in F'$, we have $C_{D/e}^*(F',f)=C_D^*(F,f)$. Since $\pi'$ is the restriction of $\pi$ to $E-e$, indeed the semi-activity of $f$ does not change.
As $e$ is the minimal element in $\pi$, it is also the minimal element in $C_D^*(F,e)$, wherefore it is internally semi-active.
This completes the proof of Theorem \ref{thm:h_of_contraction}.
\end{proof}

The formula \eqref{eq:h*D}, which was key to the proof, also enables us to establish the multiplicativity of $h^*_D$.

\begin{prop}\label{prop:szorzat}
    For any two directed graphs $D_1$ and $D_2$, and their disjoint union $D_1\sqcup D_2$, we have $h^*_{D_1\sqcup D_2}=h^*_{D_1}h^*_{D_2}$.
\end{prop}

\begin{proof}
    Let us fix arbitrary orderings $\pi_1$ and $\pi_2$ on the respective edge sets of the two graphs and concatenate them to the ordering $\pi_0$ of $E(D_1\sqcup D_2)=E(D_1)\sqcup E(D_2)$, say in such a way that the edges of $D_1$ are all smaller than the edges of $D_2$. Just like in the previous proof, take the weight functions $w_i$ for $i=0,1,2$ with $w_i(e)=- 2^{-\pi_i(e)-1}$. Induce the dissecting forest sets $\mathcal F_1=\tree(D_1, \cir^{w_1})$, $\mathcal F_2=\tree(D_2, \cir^{w_2})$, and $\mathcal F_0=\tree(D_1\sqcup D_2, \cir^{w_0})$, respectively, for the graphs $D_1$, $D_2$, and $D_1\sqcup D_2$. Let us spell out how forests $F\in \mathcal F_i$ are characterized by their fundamental cycles ($i\in\{0,1,2\}$):
\begin{multline}
\label{eq:fundamental_circuits_of_T}
\text{For $f\notin F$, either the arc of $C(F,f)$ containing $f$ has more edges than the}\\ \text{opposite arc, or the two arcs contain the same number of edges and}\\ \text{the minimal edge of the cycle, according to $\pi_i$, is along the arc opposite to $f$.}
\end{multline}
This is indeed necessary and sufficient for a forest to be compatible with $\sigma=\cir^{w_i}$ because the terms $- 2^{-\pi_i(g)-1}$ in the weights $w_i(g)$ are negligible compared to the term $1$, whence if the numbers of edges on the two arcs are different, then the arc with more edges will have the higher weight. On the other hand if the two arcs have an equal number of edges, then the terms $-2^{-\pi_i(g)-1}$ decide which arc has larger weight. This is the smallest for the edge with the smallest $\pi_i(g)$ value, and since the rest of the edges cannot overcome this, the arc containing the edge with the smallest $\pi_i$-value has the smaller weight.

From this it is obvious that $\mathcal F=\{ F_1\sqcup F_2\mid 
 F_1\in\mathcal{F}_1, F_2\in\mathcal{F}_2\}$.
Furthermore, for any $F_1\in\mathcal F_1$ and $F_2\in\mathcal F_2$, the various internal semipassivities satisfy
\[p_{\pi_0}(F_1\sqcup F_2)=p_{\pi_1}(F_1)+p_{\pi_2}(F_2),\]
because the fundamental cut, with respect to $F_1\sqcup F_2$, of any edge is contained either in $D_1$ or in $D_2$. From this and \eqref{eq:h*D}, the product formula follows immediately.
\end{proof}

\section{The case of equality}
\label{sec:equality}

We turn to examining the cases when a graph minor of $D$ inherits $h^*_D$. The relevant statements were given in the introduction.
After proving them, we discuss several situations in which equality does or does not hold.

\begin{proof}[Proof of Proposition \ref{prop:equality_in_D-e}]
    By Claim \ref{cl:loop_and_parallel_edge}, if $e$ is a loop or it has a parallel copy, then $\tcQ_{D-e}=\tcQ_D$, from which $h^*_D=h^*_{D-e}$ is obvious.

    If $e$ is a bridge, then $\tcQ_D$ is a coning over $\tcQ_{D-e}$ with apex $\mathbf{x}_e$. As a unimodular simplex $\tcQ_F$ for a spanning forest $F$ of $D-e$ stays unimodular when the vertex $\mathbf{x}_e$ is added to it, we again conclude that $h^*_{D-e}=h^*_{D}$.
    
    Conversely, by Proposition \ref{prop:root_poly_dim}, we have $\dim(\tcQ_{D-e})=\dim(\tcQ_D)$ if and only if $e$ is not a bridge. 
    Hence if $e$ is neither a loop, nor has a parallel copy, nor is it a bridge, then $\dim(\tcQ_{D-e})=\dim(\tcQ_D)$, but also $\tcQ_{D-e}\subsetneq \tcQ_D$ because
    $\mathbf{x}_e$ is not a generator of $\tcQ_{D-e}$.  
    Thus, the volume of $\tcQ_{D-e}$ is strictly smaller than that of $\tcQ_D$, which implies $h^*_{{D-e}}\neq h^*_{D}$.
\end{proof}

\begin{proof}[Proof of Proposition \ref{prop:equality_in_D/e}]
    By Claim \ref{cl:loop_and_parallel_edge} we may suppose that $D$ does not have any loops and parallel edges.

We use the notation of the proof of Theorem \ref{thm:h_of_contraction}. It is clear that the necessary and sufficient condition of equality is that the map $F'\mapsto F'\cup e$ be not only an injection but also a surjection from $\mathcal F'$ to $\mathcal F$. In other words, the condition is that all elements of $\mathcal F$ contain $e$. We have to show that this is equivalent to the assertion on $e$ that is stated in the Proposition.
	
	We start with proving that if $\mathbf x_e$ is not included in some facet $L$ of $\tcQ_D$, where $\mathbf0\notin L$, then there needs to be a forest in $\mathcal F$ that does not contain $e$. Indeed, take a generic point $\mathbf{p}$ in the relative interior of $L$. Then there is a unique forest $F\in\mathcal F$ such that $\mathbf{p}\in\tcQ_F$. 
 The point $\mathbf p$ is interior to a facet of $\tcQ_F$, which requires all but one vertex of $\tcQ_F$ to lie along $L$. Since $\mathbf0\notin L$, this means that for each edge $f\in F$, the vector $\mathbf{x}_f$ is on the facet $L$, and because $\mathbf x_e\notin L$, this implies $e\notin F$. 

	Next, we show that if $\mathbf{x}_e$ is contained by each facet that does not contain $\mathbf{0}$, then $e\in F$ for each $F\in\mathcal F$.
	By Proposition \ref{prop:facets}, the vector $\mathbf{x}_e$ is in each facet not containing $\mathbf{0}$ if and only if $\ell\cdot \mathbf{x}_e = 1$ for each admissible layering $\ell$.
	
	Suppose for a contradiction that there exists $F\in \mathcal F$ such that $e\notin F$. Fix the value $\ell=0$ for an arbitrary collection of vertices, one from each connected component of $D$. 
 This can be extended in a unique way to a vector $\ell\in\R^V$ with $\ell\cdot \mathbf{x}_f = 1$ for each $f\in F$. We claim that $\ell$ is an admissible layering.
 Tight edges form a spanning subgraph because they include $F$. We also need to show $\ell\cdot \mathbf{x}_g\leq 1$ for each $g\in E(D)-F$. To see this, note that $g\in \sigma(C(F,g))^+$ by the definition of $\mathcal F$, which indeed implies that 
 \begin{multline*}
 \ell\cdot\mathbf{x}_g = \ell\cdot\left(\sum_{f\in \sigma(C(F,g))^-}\mathbf{x}_f - \sum_{f\in F\cap\sigma(C(F,g))^+} \mathbf{x}_f\right)\\
 =|\sigma(C(F,g))^-|-(|\sigma(C(F,g))^+|-1)\leq1
 \end{multline*}
by the long arc positivity of $\sigma$.
    This proves that $\ell$ is admissible, which implies $\ell\cdot \mathbf{x}_e=1$ by our assumption on $e$. 
    As we also have $e\in \sigma(C(F,e))^+$ by the definition of $\mathcal F$. By the above computation, we have to have
    $$|\sigma(C(F,e))^+|=|\sigma(C(F,e))^-|.$$ But then, since $\pi(e)=1$, the definitions of $w$ and $\sigma=\cir^w$ imply that $e\in \sigma(C(F,e))^-$, a contradiction.
\end{proof}

The equivalent condition of Proposition \ref{prop:equality_in_D/e}, for $h^*_{D/e}=h^*_D$, is given in the language of polytopes. While it can be rephrased in graph-theoretical terms using admissible layerings (see Proposition \ref{prop:facets}), that would not be a particularly appealing condition. 
There is, however, a simpler sufficient condition that can easily be formulated in terms of graphs only.

\begin{prop}\label{prop:elegseges}
Let $D$ be a directed graph and let $e$ be an edge of $D$ so that for all cycles through $e$, the arc containing $e$ is at most as long as the opposite arc. Then we have $h^*_{D/e}=h^*_D$. In particular, if $e$ is a bridge then $h^*_{D/e}=h^*_D$.
\end{prop}

\begin{proof}
    The condition on $e$ implies that it is not a loop edge. By Proposition \ref{prop:equality_in_D/e}, it suffices to check that $\mathbf{x}_e$ is contained by each facet $L$ of $\tcQ_D$ so that $\mathbf{0}\notin L$. Proposition \ref{prop:facets} tells us that fixing such a facet is equivalent to choosing an admissible layering $\ell\colon V(D)\to\Z$. That is, for each such $\ell$, we have to show that $\ell\cdot\mathbf x_e=1$.

    Suppose that for some $\ell$ this is not so. Then by the admissibility of $\ell$, on the one hand, we have $\ell\cdot\mathbf x_e\leq0$; on the other hand there is a path $P$ in $D$, between the two endpoints of $e$, so that for each edge $f$ along $P$, we have $\ell\cdot\mathbf x_f=1$. Now, for the two arcs $A\ni e$ and $B\not\ni e$ of the cycle $P\cup e$, we obtain
    \[
    |A|-|B|\geq \left(\sum_{f\in A}\ell\cdot\mathbf x_f-\ell\cdot\mathbf x_e+1\right)-\sum_{f\in B}\ell\cdot\mathbf x_f
    =0-\ell\cdot\mathbf x_e+1\geq1,
    \]
    where we use the fact that $\sum_{f\in P\cup e} \ell\cdot \mathbf{x}_e=\ell\cdot \sum_{f\in P\cup e} \mathbf{x}_e=\ell\cdot \mathbf{0}=0$ because $P\cup e$ is a cycle.
    We have obtained that $|A|>|B|$, which contradicts our assumption that $e$ not is contained in the arc of larger cardinality for any cycle.

    As a bridge is not part of any cycle, it vacuously satisfies our condition. 
\end{proof}

\begin{coroll}\label{cor:semibalanced}
    For any edge $e$ of a semi-balanced graph $D$, we have $h^*_{D/e}=h^*_D$.
\end{coroll}

\begin{proof}
    By definition, the two arcs of each cycle of $D$ have equal length. Thus the sufficient condition of Proposition \ref{prop:elegseges} is automatically satisfied.
\end{proof}

In fact, in the case of Corollary \ref{cor:semibalanced}, it is not hard to show the stronger statement that the root polytope of $D$ is unimodularly equivalent to the extended root polytope of $D/e$.

\begin{ex}\label{ex:feloldasok}
    The triangle graph of Figure \ref{fig:triang}
    may be obtained by contracting $e_4$ in either of the semi-balanced quadrangles of Figure \ref{fig:feloldasok}. The root polytopes of both graphs are rectangles (and the extended root polytopes are cones over them), cf.\ the last panel of Figure \ref{fig:triang}. All three graphs have interior polynomial $1+x$.
\end{ex}

\begin{figure}
	\begin{tikzpicture}[scale=.80]
	\node [] (0) at (-0.8,0.7) {\small $e_1$};
	\node [] (0) at (1,-0.3) {\small $e_2$};
	\node [] (0) at (2.9,0.7) {\small $e_3$};
    \node [] (0) at (1,1.8) {\small $e_4$};
	\node [circle,scale=.6,draw] (1) at (-0.5,0) {};
	\node [circle,scale=.6,draw] (2) at (-.5,1.5) {};
	\node [circle,scale=.6,draw] (4) at (2.5,0) {};
    \node [circle,scale=.6,draw] (3) at (2.5,1.5) {}; 
	\path [thick,->,>=stealth] (1) edge [above] node {} (2);
	\path [thick,->,>=stealth] (1) edge [below] node {} (4);
	\path [thick,->,>=stealth] (3) edge [above] node {} (4);
    \path [thick,->,>=stealth] (3) edge [above] node {} (2);
 \begin{scope}[shift={(5,0.8)}]
    \node [] (0) at (0,0.8) {\small $e_1$};
	\node [] (0) at (0,-0.8) {\small $e_2$};
	\node [] (0) at (2,0.8) {\small $e_3$};
    \node [] (0) at (2,-0.8) {\small $e_4$};
	\node [circle,scale=.6,draw] (1) at (-0.5,0) {};
	\node [circle,scale=.6,draw] (2) at (1,1) {};
	\node [circle,scale=.6,draw] (4) at (2.5,0) {};
    \node [circle,scale=.6,draw] (3) at (1,-1) {};
	\path [thick,->,>=stealth] (1) edge [above] node {} (2);
	\path [thick,->,>=stealth] (3) edge [below] node {} (4);
	\path [thick,->,>=stealth] (2) edge [above] node {} (4);
 \path [thick,->,>=stealth] (1) edge [above] node {} (3);
	\end{scope}
        \end{tikzpicture}
\caption{Illustration for Example \ref{ex:feloldasok}.}\label{fig:feloldasok}
\end{figure}

\begin{remark}
    Proposition \ref{prop:equality_in_D/e}, or even just Corollary \ref{cor:semibalanced}, gives us a way to construct many different graphs with the same interior polynomial. For instance, 
    for a semi-balanced digraph $D$, the graphs $\{ D/e \mid e\in E(D)\}$ all have the same interior polynomial. 
\end{remark}

The situation described in Proposition \ref{prop:elegseges} is not the only one in which $h^*_{D/e}=h^*_D$ holds. In other words, the sufficient condition of the Proposition is not necessary. This is one thing we can learn from the following curious observation.

\begin{ex}\label{ex:teljes}
    The complete bipartite graph $K_{n,n}$ (with its so called \emph{standard orientation}, from one color class to the other) and the bidirected complete graph $K_n$ share the interior polynomial 
    \[q(x)=1+(n-1)^2x+\cdots+{n-1\choose i}^2x^i+\cdots+x^{n-1},\]
see \cite{hiperTutte,KP_Ehrhart} and \cite{arithm_symedgepoly}, respectively. For any perfect matching $M$ in $K_{n,n}$, the contraction 
$K_{n,n}/M$ is $K_n$. Therefore by Theorem \ref{thm:h_of_contraction}, if in $K_{n,n}$ we contract the edges in any subset of $M$, the resulting graph also has $q(x)$ for interior polynomial.

This holds despite the fact that if $n\geq4$, then after contracting any two elements of $M$, any of the remaining edges of $M$ is part of the longer (three-element) arc of a five-cycle. 
\end{ex}

So far in this section we started from some graph and looked for minors (specifically, edge contractions) that have the same interior polynomial. 
One may wonder about reversing this logic and looking for not smaller but bigger graphs with the same polynomial. 
We have not been overly successful at this, yet we will present some speculation to indicate the type of difficulty that arises.
Since Proposition \ref{prop:equality_in_D/e} offers far more interesting options than \ref{prop:equality_in_D-e}, we will now consider separating a vertex into a pair of vertices, assigning each incident edge to one of the two, and connecting the two new vertices with a new directed edge.

For instance, one may be tempted to generalize Example \ref{ex:teljes} by starting with a directed graph on $n$ vertices, doubling every vertex to an `upper' and a `lower' copy, connecting them by an edge from lower to upper, and lifting all other edges from the lower copy of their startpoint to the upper copy of their endpoint. The result is a bipartite graph on $n+n$ vertices, with standard orientation.
Such an operation may be worthwhile to study, however in general it will not preserve the interior polynomial, not even if we start from a bidirected graph. 

\begin{ex}
\label{ex:4-ciklus}
The bidirected four-cycle $C_4$ has the interior polynomial $1+5x+5x^2+x^3$, while the corresponding bipartite graph is the edge graph of the three-dimensional cube and has $1+5x+9x^2+x^3$.

In fact, in the case of $C_4$, it is not possible to separate even one vertex without changing the interior polynomial. In \cite{fokszam} we computed the degree of the interior polynomial and the result is stated as Theorem \ref{thm:degree} below. According to the formula, if the number of vertices increases from $4$ to $5$, the value of $\nu$ needs to increase from $0$ to $1$ too, for otherwise the degree of $h^*$ would change. In particular, our vertex needs to be separated in such a way that the resulting graph is not strongly connected anymore. There are two ways to do this: either by replacing one edge by a non-directed path of length $2$, or by creating two vertices of degree $3$, one of which is a sink and the other a source. In the former case the interior polynomial becomes $1+5x+8x^2+2x^3$, and in the latter $1+5x+7x^2+x^3$.
(For the reason why the coefficient of $x$ persists at $5$, see \cite[page 2]{fokszam}.)
\end{ex}

\begin{remark}
The observation of Example \ref{ex:4-ciklus} generalizes to the claim that if $G$ is a $2$-connected bidirected bipartite graph, then none of its vertices can be separated without changing the interior polynomial. (A cut-vertex, on the other hand, can always be separated, cf.\ the last claim of Proposition \ref{prop:elegseges}.) This is because the indicator function of either color class is an admissible layering for $G$, and it remains so after separating a vertex, too. The new edge, however, is not tight with respect to this layering, that is, the condition of Proposition \ref{prop:equality_in_D/e} (see also Proposition \ref{prop:facets}) is not met.
\end{remark}

In the next section we present a different generalization of Example \ref{ex:teljes}, one that does not fail.

\section{Gorenstein extended root polytopes}
\label{sec:gorenstein}

In this final section, we examine another important class of cases in which contracting some edges leads to no change in the interior polynomial. 
This will also be an instance of the phenomenon, discovered by Batyrev and Nill \cite{BN_comb_mirror_sym}, 
that certain projections of Gorenstein polytopes are reflexive and possess the same $h^*$-polynomial.

Gorenstein polytopes generalize the well known class of reflexive polytopes. For any polytope $P$ that is full dimensional in the real inner product space $U$ and contains the origin as an interior point, we set its \emph{dual polytope} to be 
\[P^*=\{u\in U\mid\langle u,x\rangle\geq-1\text{ for all }x\in P\}.\]

\begin{defn}
    A lattice polytope, containing $\mathbf{0}$ in its relative interior, is called \emph{reflexive} if its dual, with respect to the restriction of the standard dot product to the linear span of the polytope, is also a lattice polytope.
\end{defn}

In a reflexive polytope, the origin is the unique interior lattice point. By Hibi's result \cite[Theorem 4.6]{BeckRobbins}, a lattice polytope is reflexive if and only if its $h^*$-polynomial is palindromic with a degree that coincides with the dimension of the polytope.
Higashitani \cite[Proposition 1.4]{smooth_Fano} showed that $\tcQ_D$ is reflexive if and only if $D$ is totally cyclic, that is, each weakly connected component of $D$ is strongly connected. 

\begin{defn}[Gorenstein polytope \cite{BN_comb_mirror_sym}]
    A 
    lattice polytope $P$, 
    with respect to the positive integer $r$, is called a \emph{Gorenstein polytope of index $r$} if the dilation $rP$ contains a (relative) interior lattice point $\mathbf{p}$ so that
    the translation $rP-\mathbf{p}$ is a reflexive polytope.
\end{defn}

Reflexive polytopes are Gorenstein of index $1$. 
Hibi \cite{hibi_gorenstein} also 
proved that the lattice polytope $P$ is Gorenstein of index $r$ if and only if its $h^*$-polynomial is palindromic of degree $\dim P-r+1$ (see \cite{BrunsRoemer,Stanley78} as well). 
In particular, the index of a Gorenstein polytope is unique. 

Returning now to the specific theme of this paper,
we give the following characterization of Gorenstein extended root polytopes: 

\begin{thm}\label{thm:Gorenstein_root_poly_char}
Let $D$ be a directed graph. The extended root polytope $\tcQ_D$ is Gorenstein if and only if there exists a set
$K$ of edges of $D$ such that 
	\begin{enumerate}[label=(\roman*)]
		\item \label{egy} each elementary directed cut of $D$ contains exactly one edge of $K$,
		\item \label{ketto} each admissible layering $\ell$ of $D$ has $\ell(h)-\ell(t)=1$ for each edge $\overrightarrow{th}\in K$.
	\end{enumerate}
Moreover, in this case $K$ is a minimal cardinality dijoin, the index of $\tcQ_D$ is $|K|+1$, furthermore, all minimal cardinality dijoins $K$ of $D$ satisfy \ref{egy} and \ref{ketto}. 
\end{thm}

Before giving the proof, let us recall two results from \cite{fokszam}. These were proved for weakly connected graphs and here we generalize them for arbitrary digraphs.

\begin{thm}\label{thm:degree}
    For a directed graph $D$, the degree of $h^*_{D}$ is equal to $|V|-c(D)-\nu(D)$, where $c(D)$ is the number of weakly connected components of $D$ and $\nu(D)=\min\{|K| \mid K\text{ is a dijoin of }D \}$.
\end{thm}

\begin{proof}
    By \cite[Theorem 1.1]{fokszam}, for a connected digraph $D$, the degree of $h^*_{D}$ is equal to $|V|-1-\nu(D)$. Note that in general, $K$ is a dijoin in $D$ if and only if it is a union of dijoins taken from each weakly connected component. Now
    Proposition \ref{prop:szorzat} implies the claim.
\end{proof}

\begin{prop}\label{prop:leading_coefficient}
    For a directed graph $D$, the leading coefficient of $h^*_D$ is equal to the number of vectors that can be obtained as $\sum_{e\in K} \mathbf{x}_e$ for a minimal cardinality dijoin $K$ of $D$.
\end{prop}

\begin{proof}
    For a weakly connected graph $D$, \cite[Theorem 1.3]{fokszam} claims exactly the above statement. The general case follows by Proposition \ref{prop:szorzat}.
\end{proof}

\begin{proof}[Proof of Theorem \ref{thm:Gorenstein_root_poly_char}]
    By Theorem \ref{thm:degree}, the degree of $h^*_{D}$ is $|V|-c(D)-\nu(D)$.
    As the dimension of $\tcQ_D$ is $|V|-c(D)$, if it is Gorenstein, then it is Gorenstein of index $\nu(D)+1$, cf.\ 
    Hibi's result 
    \cite{hibi_gorenstein} quoted above. 
    In other words, $\tcQ_D$ being Gorenstein is equivalent to the statement that 
    $(\nu(D)+1)\tcQ_D$ contains 
    an interior lattice point $\mathbf{p}$, such that the translation $(\nu(D)+1)\tcQ_D-\mathbf{p}$ is a reflexive polytope. 

    By \cite[Theorem 1.3]{fokszam} the interior lattice points of $(\nu(D)+1)\tcQ_D$ are of the form $\sum_{e\in K} \mathbf{x}_e + \mathbf{0}$, where $K$ is a dijoin of the minimum cardinality $\nu(D)$. 
    In \cite{fokszam}, $D$ was assumed to be weakly connected, but the argument carries over to the general case. Alternatively, one can invoke \cite[Theorem 1.8]{fokszam}, which establishes the same claim for all regular matroids --- at that level of generality, connectedness is not an issue anymore.
    
    Let us then examine when $(|K|+1)\tcQ_D - \sum_{e\in K} \mathbf{x}_e$ is reflexive for some 
    subset $K$ of $E(D)$.
    Note that for a directed cut $C^*$ in $D$, we have $f_{C^*}(\sum_{e\in K} \mathbf{x}_e)=|C^*\cap K|$, as each edge  $e\in C^*\cap K$ contributes $f_{C^*}(\mathbf{x}_e)=1$, whereas $f_{C^*}(\mathbf{x}_e)=0$ for all $e\notin C^*$.
    Hence, by Proposition \ref{prop:facets}, the defining inequalities corresponding to the facets of $(|K|+1)\tcQ_D - (\sum_{e\in K} \mathbf{x}_e)$ are as follows:
\begin{itemize}
    \item For each elementary cut $C^*$ in $D$, we have 
    $-f_{C^*}\cdot \mathbf x\leq |C^* \cap K|$.
    \item For each admissible layering $\ell$ of $D$, we have $\ell\cdot\mathbf x\leq |K|+1-\sum_{e\in K}\ell\cdot \mathbf{x}_e$. 
    \end{itemize}
    As $f_{C^*}$ has coordinates that are 1, and $\ell$ has coordinates whose difference is 1, 
    the dual polytope has integer vertices if and only if $|C^* \cap K|=1$ for each elementary directed cut $C^*$, and $|K|+1-\sum_{e\in K}\ell\cdot \mathbf{x}_e=1$ for each admissible layering $\ell$. The latter condition is equivalent to requiring
    $\ell\cdot\mathbf x_e=1$ for each admissible layering $\ell$ and edge $e\in K$.
    
    Thus indeed, on the one hand, if $\tcQ_D$ is Gorenstein then there exists a dijoin $K$
    (necessarily of minimum cardinality)
    satisfying \ref{egy} and \ref{ketto}. On the other hand, if $K$ is a subset of $E(D)$ that satisfies \ref{egy} and \ref{ketto}, then $\sum_{e\in K} \mathbf{x}_e + \mathbf{0}\in(|K|+1)\tcQ_D$ is a lattice point 
    so that $(|K|+1)\tcQ_D - \sum_{e\in K} \mathbf{x}_e$ is reflexive; in particular $\tcQ_D$ is Gorenstein of index $|K|+1$. Now by \ref{egy} $K$ is obviously a dijoin, furthermore the uniqueness of the index implies that $|K|=\nu(D)$, i.e., that $K$ is a minimum cardinality dijoin. 
    
    We also need to show that if $\tcQ_D$ is Gorenstein, then each minimal cardinality dijoin $K$ of $D$ satisfies \ref{egy} and \ref{ketto}. As in this case, by palindromicity, the leading coefficient of $h^*_D$ is $1$, Proposition \ref{prop:leading_coefficient} implies that the vector $\sum_{e\in K} \mathbf{x}_e$ is the same internal lattice point of $(\nu(D)+1)\tcQ_D$ for each minimal cardinality dijoin\footnote{This is closely related to the last claim of \cite[Proposition 2.8]{BN_comb_mirror_sym} via Proposition \ref{prop:special_simplices_in_root_polytopes} below.}.
    Hence the above reasoning can be applied to an arbitrary minimal cardinality dijoin, to deduce that \ref{egy} and \ref{ketto} are satisfied.
\end{proof}

The characterization in Theorem \ref{thm:Gorenstein_root_poly_char}, of Gorenstein extended root polytopes, is far less natural from the point of view of graph theory than that of reflexive (extended root) polytopes. 
Still, 
some well known
results about Gorenstein polytopes have nice graph-theoretic interpretations in the root polytope setting.
One important notion in connection with Gorenstein polytopes is that of a special simplex:

\begin{defn}[special $(r-1)$-simplex \cite{BN_comb_mirror_sym}]
	Let $P$ be a 
    lattice polytope. A simplex $S$ spanned by $r$ affinely independent lattice points in $P$ is called a \emph{special $(r-1)$-simplex} of $P$ if each facet of $P$ contains exactly $r-1$ vertices of $S$.
\end{defn}

According to \cite[Proposition 3]{Sti98}, if $P$ is Gorenstein of index $r$ and has a regular unimodular triangulation, then it also has a special $(r-1)$-simplex.
In our context, one source of special simplices is the following. 

\begin{prop}\label{prop:special_simplices_in_root_polytopes}
If $\tcQ_D$ is a Gorenstein polytope (for a digraph $D$), and $K$ is a minimal cardinality dijoin in $D$, then 
$\tcQ_K$
is a special $\nu(D)$-simplex in $\tcQ_D$. 
\end{prop}

\begin{proof}
Proposition \ref{prop:facets} describes the facets of $\tcQ_D$ in terms of elementary directed cuts and admissible layerings.
By Theorem \ref{thm:Gorenstein_root_poly_char}, if $\tcQ_D$ is Gorenstein, then any minimal cardinality dijoin $K$ in $D$ is such that each elementary directed cut of $D$ contains exactly one edge of $K$, and each admissible layering $\ell$ of $D$ has $\ell(h)-\ell(t)=1$ for each edge $\overrightarrow{th}\in K$. 

This ensures that each facet corresponding to an admissible layering $\ell$ contains $\mathbf{x}_e$ for each $e\in K$. As it does not contain $\mathbf{0}$, facets corresponding to an admissible layering satisfy the requirement.

For a facet corresponding to an elementary directed cut $C^*$, the origin 
is in the facet, as well as $\mathbf{x}_e$ for $e\notin C^*$, while $\mathbf{x}_e$ for $e\in C^*$ is not in the facet. Thus also for any of these facets, there is exactly one vertex of $\tcQ_K$ that is not contained in it. 
\end{proof}

We note that not all special simplices in $\tcQ_D$ are of the above form. For example, for the (Gorenstein) digraph of Figure \ref{fig:triang}, the simplex spanned by $\mathbf{x}_{e_1}$ and $\mathbf{x}_{e_3}$ is also special, besides $\conv\{\mathbf0,\mathbf{x}_{e_2}\}$.

Special simplices are important because of 
the following theorem of Batyrev and Nill \cite{BN_comb_mirror_sym}. 

\begin{thm}\cite[Theorem 2.16]{BN_comb_mirror_sym} 
\label{thm:Gorenstein_special_projection}
Let $P$ be a $d$-dimensional Gorenstein polytope of index $r$, and suppose that $\Delta$ is a special $(r-1)$-simplex. Let $Q$ be the projection of $P$ along the affine hull of $\Delta$. Then $Q$ is a $(d-r+1)$-dimensional reflexive polytope, and $h^*_Q = h^*_P$.
\end{thm}

In particular, for extended root polytopes we get the following statement. For the convenience of the reader we include a new proof, specific to directed graphs.

\begin{thm}\label{thm:Gorenstein_special_projection_root_polytope}
Suppose that the extended root polytope $\tcQ_D$  of the digraph $D$ has the Gorenstein 
property. Let $K$ be a minimal cardinality dijoin of $D$, that is, $|K|=\nu(D)$. Then $\tcQ_{D/K}$ is a $(|V|-c(D)-\nu(D))$-dimensional reflexive polytope, furthermore $h^*_{D/K} = h^*_D$.
\end{thm}

The easy examples of Figures \ref{fig:triang} and \ref{fig:feloldasok} are all Gorenstein digraphs. In the case of the triangle, $\{e_2\}$ is a minimal dijoin; in either quadrilateral, both $\{e_1,e_3\}$ and $\{e_2,e_4\}$ are minimal dijoins. Contracting any of these dijoins results in a graph with two vertices and two oppositely oriented edges, the interior polynomial of which is still $1+x$.

Example \ref{ex:teljes} is also a special case of Theorem \ref{thm:Gorenstein_special_projection_root_polytope}, as can be seen from the palindromicity of $q(x)$. The fact that $M$ is a dijoin is apparent from the quotient graph $K_n=K_{n,n}/M$ being strongly connected (which is also the reason for the reflexivity of the quotient polytope); that $M$ is minimal follows from the necessity of covering all star-cuts. 

\begin{proof}[Proof of Theorem \ref{thm:Gorenstein_special_projection_root_polytope}]
    
    As $\tcQ_D$ is Gorenstein, Theorem \ref{thm:Gorenstein_root_poly_char} guarantees that it has index $\nu(D)+1$. For a minimal cardinality dijoin $K$ of $D$, each elementary directed cut contains exactly one edge of $K$ and each admissible layering $\ell$ of $D$ has $\ell(h)-\ell(t)=1$ for each $\overrightarrow{th}\in K$.
    
    The projection of $\tcQ_D$ along $\tcQ_K$ is $\tcQ_{D/K}$ (where $D/K$ is the digraph obtained from $D$ by contracting the edges in $K$). Since $K$ is a dijoin, $D/K$ has no directed cuts. In other words, each weakly connected component of $D/K$ is strongly connected. Thus, by \cite{smooth_Fano}, the polytope $\tcQ_{D/K}$ is reflexive. Moreover, since $K$ is easily seen to be cycle-free (cf.\ \cite[Lemma 2.13]{fokszam}), we have $|V(D/K)|=|V(D)|-|K|=|V(D)|-\nu(D)$. Hence the dimension of $\tcQ_{D/K}$ is 
    \[|V(D)|-\nu(D)-c(D/K)=(|V(D)|-c(D))-(\nu(D)+1)+1.\]

    It remains to show that $h^*_{D}=h^*_{{D/K}}$. For this, we will apply Proposition \ref{prop:equality_in_D/e} successively. To do so, we will need the following claim.

    \begin{claim}\label{cl:layerings_after_contraction}
    If some edges $f_1=\overrightarrow{t_1h_1}$ and $f_2=\overrightarrow{t_2h_2}$ have the property that $\ell(h_1)- \ell(t_1) = \ell(h_2)- \ell(t_2) = 1$ for each admissible layering $\ell$ of $D$, and there is no elementary directed cut containing both $f_1$ and $f_2$, then we have $\ell'(h_2)-\ell'(t_2)=1$ for each admissible layering $\ell'$ of $D/{f_1}$.   
    \end{claim}
    
    \begin{proof}
        For any function $\ell\colon V\to \mathbb{Z}$, let us call $\ell(h)-\ell(t)$ the 
        \emph{jump} of the edge $\overrightarrow{th}$. As before, if the 
        jump of an edge is $1$, we call it a tight edge with respect to $\ell$.
        Let $w$ be the vertex of $D/f_1$ obtained by contracting $f_1$ (that is, gluing together $t_1$ and $h_1$).
        
        Suppose for a contradiction that there exists an admissible layering $\ell'$ of $D/f_1$ such that $f_2$ is not tight. From $\ell'$, we can construct a layering $\ell$ of $D$ such that $\ell(v)=\ell'(v)$ for each vertex $v\neq w$, and $\ell(t_1)=\ell(h_1)=\ell'(w)$. Clearly, since $\ell'$ was admissible, $\ell$ satisfies the property $\ell(h)-\ell(t)\leq 1$ for each $\overrightarrow{th}$. However, since $\ell(h_1)-\ell(t_1)=0$, our assumption that $f_1$ is tight for each admissible layering implies that $\ell$ cannot be admissible. This means that the subgraph of tight edges for $\ell$ is not weakly connected, instead, because $\ell'$ is admissible, it is a subgraph of exactly two connected components. In other words, there exists 
        an elementary cut $C^*$ in $D$, 
        containing $f_1$ and
        consisting only of non-tight edges. 
        Let the shores of the cut be $V_0$ and $V_1$ with $t_1\in V_0$ and $h_1\in V_1$. 
        
        Next, consider the function  $\tilde{\ell}$ defined by
        $$
        \tilde{\ell}(v) = \left\{
        \begin{array}{cl}
	\ell(v)+1 & \text{ if } v\in V_1\\
	\ell(v) & \text{ if } v\in V_0 \\
        \end{array}\right.\ .
        $$
        This satisfies $\tilde{\ell}(h_1)-\tilde{\ell}(t_1)=1$. Moreover, the jump 
        of edges within $C^*$ is still at most $1$, while the 
        jumps of edges outside $C^*$ remain the same as in the case of $\ell$. As 
        there are no more cuts consisting of non-tight edges, 
        $\tilde{\ell}$ is an admissible layering. Since $f_1\in C^*$, and we assumed that no elementary directed cut contains both $f_1$ and $f_2$, we have $f_2\notin C^*$, wherefore $\tilde{\ell}(h_2)-\tilde{\ell}(t_2)=\ell'(h_2)-\ell'(t_2)<1$, which contradicts the assumption of the Claim on $f_2$ and the admissible layering $\tilde\ell$.
    \end{proof}

    Let $K=\{e_1, \dots , e_k\}$. Note that $D/K=(\dots ((D/e_1)/e_2)\dots)/e_k$. By Claim \ref{cl:layerings_after_contraction} and Proposition \ref{prop:equality_in_D/e}, we have $h^*_{D}=h^*_{{D/{e_1}}}=h^*_{{(D/{e_1})/e_2}}=\dots =h^*_{{D/K}}$.
\end{proof}

 In particular, by contracting edges of the minimal dijoin one by one, we find sequences of Gorenstein digraphs and corresponding polytopes that interpolate between $\tcQ_D$ and $\tcQ_{D/K}$. All interpolating polytopes have the same $h^*$-polynomial and their dimensions and Gorenstein indices both form arithmetic progressions of difference $-1$. 

\appendix
\section{Proof of Theorem \ref{thm:interior_poly_fixed_order}}

Here, we give a proof of Theorem \ref{thm:interior_poly_fixed_order}, which is a slight modification of the proof of \cite[Theorem 1.8]{fix_elsorrend} that claims the same formula for the case of a connected, bidirected $D$.

Let $P\subset\mathbb R^n$ be a lattice polytope, and let $\Delta_1, \dots, \Delta_s$ be a dissection of $P$ using no new vertices. We say that a point $\mathbf{q}$ (of the affine hull $A$ of $P$) is in \emph{general position} with respect to the dissection $\Delta_1, \dots, \Delta_s$ if $\mathbf{q}$ is not contained in any facet-defining hyperplane (with respect to $A$) of any of the simplices $\Delta_1, \dots, \Delta_s$.

For two points $\mathbf{p},\mathbf{q}\in \mathbb{R}^n$, let us denote by $[\mathbf p,\mathbf q]$ the closed segment connecting them, and let us denote by $(\mathbf{p},\mathbf{q})$ the relative interior of this segment.

We say that a point $\mathbf{p}\neq \mathbf q$ of a simplex $\Delta_i$ is \emph{visible} from $\mathbf q$ if $(\mathbf{p},\mathbf{q})$ is disjoint from $\Delta_i$. We say that a facet of $\Delta_i$ is visible from $\mathbf{q}$ if all points of the facet are visible from $\mathbf{q}$. 
It is easy to see that a facet of $\Delta_i$ is visible from $\mathbf q$ if and only if its hyperplane separates $\mathbf q$ from the (relative) interior of $\Delta_i$.

For a simplex $\Delta$, let $\Vis_{\mathbf q}(\Delta)$ be the set of facets of $\Delta$ that are visible from $\mathbf{q}$.
Let also $H_\mathbf{q}(\Delta) = \Delta - \bigcup_{F\in \Vis_{\mathbf{q}}(\Delta)} F$. I.e., we remove the visible facets from $\Delta$. 

The following Proposition from \cite{fix_elsorrend} is the key to proving Theorem \ref{thm:interior_poly_fixed_order}. 
It is motivated by \cite[Proposition 2.1]{arithm_symedgepoly}, and generalizes that result to unimodular dissections.

\begin{prop}\cite{fix_elsorrend} \label{prop:h^*-vector_from_visible_facets}
Let $\Delta_1, \dots, \Delta_s$ be a dissection of the $d$-dimensional lattice polytope $P\subset\mathbb R^n$ into unimodular simplices, and let $\mathbf{q}\in P$ be a point in general position with respect to the dissection. 
Then the $h^*$-polynomial $h^*(x)=h^*_d x^d +  \dots + h^*_1 x + h^*_0$ of $P$ 
has the coefficients
\[h^*_i = \left|\{ j \mid 1\le j\le s\text{ and } |\Vis_{\mathbf{q}}(\Delta_j)| = i\}\right|.\]
\end{prop}

\begin{remark}
Even though we no longer assume the dissection to be a triangulation, the condition of unimodularity remains, and it is crucial for the above statement to hold.
\end{remark}

\begin{proof}[Proof of Theorem \ref{thm:interior_poly_fixed_order}]
    The proof is a slight modification of the proof \cite[Theorem 1.8]{fix_elsorrend} which is in turn motivated by the proof of \cite[Proposition 4.6]{arithm_symedgepoly}, that concerns a special case: a concrete family of triangulations of symmetric edge  polytopes of complete bipartite graphs.
	The main engine of the proof is Proposition \ref{prop:h^*-vector_from_visible_facets}.
	
	Let us fix a dissecting forest set $\mathcal{F}$ for $D$ and fix an ordering of the edges of $D$ as well. We use the notation $e_1 < e_2 < \dots < e_m$, where $m=|E(D)|$.
	Our strategy is to find a point $\mathbf{q}$ in the interior of $\tcQ_D$, such that for each simplex $\tcQ_F$ of the dissection, the number of facets of $\tcQ_F$ visible from $\mathbf q$ is equal to the internal semi-passivity of the spanning forest $F$ with respect to $<$.
	
    Specifically, we choose the point
	$$\mathbf{q}= \sum_{i=1}^m \left(\frac{2^{-i}}{1+\sum_{j=1}^m 2^{-j}}\right) \mathbf{x}_{e_i}.$$
    It is a convex combination of $\mathbf0$ and the  $\mathbf{x}_e$ ($e\in E$), whence by definition $\mathbf q\in \tcQ_D$.
    In fact it is a relative interior point, which can be easily checked by plugging it into the equation of any facet, as given in Proposition \ref{prop:facets}.
	
	Next, we show that for any 
    $F\in\mathcal F$, 
    the number of facets of $\tcQ_F$, visible from $\mathbf q$, is equal to the internal semi-passivity of $F$ with respect to $<$. At the same time, it will also turn out that 
    $\mathbf q$ is in general position.
	
	Let $F\in \mathcal{F}$ be a spanning forest of the dissecting forest set. The simplex $\tcQ_F$ has $\mathcal{Q}_F$ as a facet. We claim that this facet is never visible from $\mathbf q$, for the simple reason that $\mathcal{Q}_F$ is part of the boundary of $\tcQ_D$. Indeed if it was not, then for the barycenter $\mathbf f$ (or any other relative interior point) of $\mathcal Q_F$ and a small enough $\varepsilon>0$, we would have $(1+\varepsilon)\mathbf f\in\tcQ_D\setminus\tcQ_F$, making it necessary for there to be another spanning forest $F'\in\mathcal F$ with $(1+\varepsilon)\mathbf f\in\tcQ_{F'}$. But then by definition, $(1-\varepsilon)\mathbf f$ would be both a point of $\tcQ_{F'}$ and an interior point of $\tcQ_F$, which would contradict the dissecting property of $\mathcal F$.
    
    The rest of the facets of $\tcQ_F$ are of the form $\tcQ_{F-e_k}=\conv\{ \mathbf{x}_{e_j} \mid e_j\in F-e_k\}$ for some edge $e_k\in F$.	
	We will show that the facet $\tcQ_{F-e_k}$ is visible from $\mathbf q$ if any only if $e_k$ is internally semi-passive in $F$. 
	Note that the hyperplane of $\tcQ_{F-e_k}$
	is described as the kernel of 
	the linear functional $f_{C^*(F,e_k)}$. Recall that for a cut $C^*$ with shores $V_0$ and $V_1$, the vector $f_{C^*}\in\mathbb{Z}^V$ is defined by $f_{C^*}(v)=1$ for $v\in V_1$ and $f_{C^*}(v)=0$ for $v\in V_0$. Let us suppose that $f_{C^*(F,e_k)}$ has value $1$ at the head of $e_k$ and value $0$ at its tail.
    Then $f_{C^*(F,e_k)}\cdot \mathbf{x}_{e_k} = 1$, moreover $f_{C^*(F,e_k)}\cdot \mathbf{x}_e = 0$ for $e\in F-e_k$, and $f_{C^*(F,e_k)}\cdot \mathbf{0} = 0$.
    Thus $\{\mathbf{x} \mid f_{C^*(F,e_k)}\cdot \mathbf{x}=0\}$ is indeed the hyperplane of $\tcQ_{F-e_k}$, and (as $f_{C^*(F,e_k)}\cdot \mathbf{x}_{e_k} = 1$) the facet $\tcQ_{F-e_k}$ of $\tcQ_F$ is visible from $\mathbf q$ if and only if $f_{C^*(F,e_k)}\cdot \mathbf q < 0$.
	Here by linearity, we have 
	$$f_{C^*(F,e_k)}\cdot \mathbf q= \sum_{i=1}^m \left(\frac{2^{-i}}{1+\sum_{j=1}^m 2^{-j}}\right)f_{C^*(F,e_k)}\cdot \mathbf{x}_{e_i}.$$
	Non-zero contributions to this sum come from the edges of $C^*(F,e_k)$. Namely, those $e_i$ in the cut
	that stand parallel to $e_k$ (that is, have their heads in the same shore as $e_k$) have $f_{C^*(F,e_k)}\cdot\mathbf{x}_{e_i} = 1$ and the edges $e_i$ of $C^*(F,e_k)$ standing opposite to $e_k$ 
	have $f_{C^*(F,e_k)}\cdot\mathbf{x}_{e_i} = -1$. Hence as $C^*(F,e_k)\ne\varnothing$, the sum is never $0$, in other words $\mathbf q$ is in general position with respect to the dissection. Furthermore, $f_{C^*(F,e_k)}\cdot\mathbf q<0$ 
	if and only if the smallest edge of $C^*(F,e_k)$, according to $<$, stands opposite to $e_k$, i.e., if and only if $e_k$ is internally semi-passive in $F$.
\end{proof}

\bigskip
\noindent
\textbf{Acknowledgements}
We thank Akihiro Higashitani and Max K\"olbl for stimulating conversations.
TK was supported by the Japan Society for the Promotion of Science (JSPS) Grant-in-Aid for Scientific Research C no.\ 23K03108.
LT was supported by the National Research, Development and Innovation Office of Hungary -- NKFIH, grant no.\ 132488, by the János Bolyai Research Scholarship of the Hungarian Academy of Sciences, and by the ÚNKP-23-5 New National Excellence Program of the Ministry for Innovation and Technology, Hungary. This work was also partially supported by the Counting in Sparse Graphs Lendület Research Group of the Alfr\'ed Rényi Institute of Mathematics.

\bibliographystyle{plain}
\bibliography{Bernardi}

\end{document}